\documentclass{amsart}
\usepackage{amssymb,amscd}  

\newtheorem{theorem}{Theorem}[section]
\newtheorem{lemma}[theorem]{Lemma}
\newtheorem{corollary}[theorem]{Corollary}

\newtheorem{proposition}[theorem]{Proposition}

\theoremstyle{definition}
\newtheorem{definition}[theorem]{Definition}

\newtheorem{remark}[theorem]{Remark} 
\numberwithin{equation}{section}

\newcommand\C{\mathbf{C}}

\newcommand\Z{\mathbf{Z}}
\newcommand\F{\mathbf{F}}
\newcommand\Fq{\F_q}

\newcommand\tensor{\otimes}
\newcommand\isomorphic{\cong}

\DeclareMathOperator{\Sym}{Sym}

\DeclareMathOperator{\Divisor}{div}

\newcommand\intersect{\cap}

\newcommand\abs[1]{{\left|#1\right|}}

\newcommand\Projective{{\bf P}} 

\newcommand\Pinf{{P_\infty}} 
\DeclareMathOperator{\Pic}{Pic}
\DeclareMathOperator{\Div}{Div}
\DeclareMathOperator{\Eff}{Eff}
\newcommand{\LL}{\mathcal{L}}
\newcommand{\KK}{\mathcal{K}}
\newcommand{\Jhat}{\widehat{J}}
\newcommand{\OC}{{\mathcal{O}_C}}

\begin{document}

\title{Upper bounds for some Brill-Noether loci over a finite field}

\author{Kamal Khuri-Makdisi}
\address{Mathematics Department, 
American University of Beirut, Bliss Street, Beirut, Lebanon}
\email{kmakdisi@aub.edu.lb}
\thanks{June 20, 2017}

\begin{abstract}
Let $C$ be a smooth projective algebraic curve of genus $g$ over the finite
field $\Fq$.  A classical result of H. Martens states that the
Brill-Noether locus of line bundles $\LL$ in $\Pic^d C$ with $\deg \LL = d$ 
and $h^0(C,\LL) \geq i$ is of dimension at most $d - 2i + 2$, under
conditions that hold when such an $\LL$ is both effective and special.  We
show that the number of such $\LL$ that are rational over $\Fq$ is bounded
above by $K_g q^{d - 2i + 2}$, with an explicit constant $K_g$ that grows
exponentially with $g$.  Our proof uses the Weil estimates for function
fields, and is independent of Martens' theorem.  We apply this
bound to give a precise lower bound of the form $1 - K'_g/q$ for the
probability that a line bundle in $\Pic^{g+1} C(\Fq)$ is base point free.
This gives an effective version over finite fields of the usual statement
that a general line bundle of degree $g+1$ is base point free.  This
is applicable to the author's work on fast Jacobian group arithmetic for
typical divisors on curves.
\end{abstract}

\maketitle

\section{Introduction}
\label{section1}

Let $C$ be a smooth projective algebraic curve of genus $g$.  A classical
object of study is the Brill-Noether locus of $C$ with parameters $(d,i)$.
It is the variety of degree~$d$ line bundles $\LL$ on $C$ whose space of
global sections has dimension at least~$i$:
\begin{equation}
\label{equation1.1}
\{ \LL \in \Pic^d C \mid h^0(\LL) \geq i \}.
\end{equation}
Here $h^0(\LL) = \dim H^0(C,\LL)$.
For ``uninteresting'' values of $(d,i)$, Riemann-Roch and other
considerations such as Clifford's theorem on special divisors imply that
the above set is either empty or equal to all of $\Pic^d C$.  As we shall
recall in Section~\ref{section2}, the ``interesting'' range is when 
\begin{equation}
\label{equation1.2}
0 \leq i-1 \leq d-i+1 \leq g-1.
\end{equation}
When~\eqref{equation1.2} holds, one has upper and lower bounds on
the dimension of the Brill-Noether locus; see for example Chapters IV and~V
of~\cite{ACGH}.  The main aspect that we will consider in this article is
the upper bound on the dimension.  In this regard, there is a basic theorem
of Martens~\cite{Martens}; see Theorem~IV.5.1 of~\cite{ACGH}.  It says that
the dimension of the Brill-Noether locus in the interesting range is
bounded above by $d-2i+2$.  In fact, Martens' theorem is somewhat sharper,
as it says that the dimension $d-2i+2$ is attained if and only if the curve
$C$ is hyperelliptic, so the upper bound for a nonhyperelliptic curve is in
fact $d-2i+1$.  Sharper bounds are known if one excludes not only
hyperelliptic curves, but also other special cases; see for example
Mumford's result in Theorem~IV.5.2 of~\cite{ACGH}.  However, our main
interest here is in a bound that applies uniformly for all genus $g$
curves over a finite field.

The main result of this note is a quantitative version of the dimension 
bound $d-2i+2$ over a finite field $\Fq$.  We obtain the result,
Theorem~\ref{theorem2.8} below, that the number of $\Fq$-rational
points in the Brill-Noether locus~\eqref{equation1.1} is bounded above by
$K q^{d-2i+2}$, where $K=16^g$ is a constant that depends only on the genus
$g$.  Applying this result to the $\F_{q^a}$-rational points as $a \to
\infty$, one recovers Martens' result, at least when the base field is
finite, because a dimension $e$ variety over $\Fq$ has 
$\Theta(q^{ae})$ points over $\F_{q^a}$.
Our quantitative result also incidentally implies an alternate
proof of Clifford's theorem over a finite field.  Our main tool in the
proof is the Weil bounds for both zeta and L-functions of the curve $C$.
We use the Weil bounds in Theorem~\ref{theorem2.6} to bound a larger set
$X_d$, defined below in~\eqref{equation2.5}, that is related to the various
Brill-Noether loci with fixed $d$ and varying $i$.  Our results do not seem
to yield lower bounds; it would be interesting to see if any of the lower
bounds on the dimensions on Brill-Noether loci can be proved by counting
points over finite fields in some way.

We apply our result on upper bounds for Brill-Noether loci to
obtain a precise estimate of the probability that a random element
$\LL \in \Pic^{g+1} C$ is base point free; here the random element is
drawn uniformly from $(\Pic^{g+1} C)(\Fq)$.  We similarly (and more easily)
obtain a precise estimate of the probability that a random element 
$\LL \in (\Pic^{g-1} C)(\Fq)$ has $h^0(\LL) = 0$.  For both of these
questions, it is easy to see that the dimension of the subvarieties of
$\Pic^{g \pm 1} C$ where the desired condition does \emph{not} hold is
at most $g-1$, from which it follows that the probabilities are at
least $1 - A/q$ for a suitable constant $A$.  However, finding the
precise $A$ from just the dimension count and some kind of estimate of
the degree seems elusive.  We can moreover use the more precise 
estimates here to obtain a bound for the probability that a divisor
is ``typical'' for fast Jacobian group arithmetic, in the sense of our
preprint~\cite{KKMtypical}.  Obtaining that bound was the main
motivation for this article.

\section{The main result}
\label{section2}

We first review why~\eqref{equation1.2} is the range in which pairs
$(d,i)$ can potentially give an interesting Brill-Noether
locus~\eqref{equation1.1}.  Let $\KK$ denote the canonical bundle on
$C$, so that $h^1(\LL) = \dim H^1(C, \LL) = h^0(\KK \tensor
\LL^{-1})$.  Then, by Riemann-Roch, $h^0(\LL) = h^1(\LL) + d + 1 - g$,
whenever $\deg \LL = d$.  Since we always have $h^0(\LL), h^1(\LL)
\geq 0$, it follows that if either $i \leq 0$ or $i \leq d+1-g$, the
Brill-Noether locus of~\eqref{equation1.1} is all of $\Pic^d C$, hence
uninteresting.  So the interesting case can only occur when both
$i \geq 1$ and $i \geq d+2-g$ are satisfied; these are precisely the
outer inequalities of~\eqref{equation1.2}.  Moreover, when these
outer inequalities hold, we have both $h^0(\LL) \geq 1$ and
$h^1(\LL) \geq 1$, so the line bundle $\LL$ corresponds to an
effective special divisor; in this setting, Clifford's theorem
(Theorem IV.5.2 of~\cite{Hartshorne}) says that $i-1 \leq d/2$, which
gives the middle inequality of~\eqref{equation1.2}.  Actually,
our reasoning does not need the inequality arising from Clifford's
theorem.  The purist reader may prefer not to assume this
inequality, but may rather deduce it from Theorem~\ref{theorem2.8},
which gives an alternate proof of Clifford's theorem over a finite
field; see Remark~\ref{remark2.9.5}.

Our second preliminary observation is that we can restrict $d$ to the
interval $0 \leq d \leq g-1$ for the purpose of bounding the size of
the interesting Brill-Noether loci.  First, elementary considerations
about degrees of $\LL$ and $\KK \tensor \LL^{-1}$ tell us that
$h^0(\LL)$ is predicted entirely by $d$, hence uninteresting, except
in the range $0 \leq d \leq 2g-2$.  Furthermore, we have a bijection
between $\Pic^d C$ and $\Pic^{2g-2-d} C$ sending $\LL$ to $\KK \tensor
\LL^{-1}$, and this bijection (which is an involution on the set of
all line bundles on $C$) replaces the Brill-Noether locus for the
pair $(d,i)$ with the locus for the pair 
$(2g-2-d, i-d-1+g)$, while leaving the quantity $d-2i+2$ unchanged.
This allows us to reduce the interval for $d$ to half its size.  This
is indeed the assumption made in Theorem~\ref{theorem2.6} and 
case~(i) of Theorem~\ref{theorem2.8} below, namely:
\begin{equation}
\label{equation2.0}
0 \leq d \leq g-1, \qquad i \geq 1.
\end{equation}
Note than when~\eqref{equation2.0} holds, then $d-i+1 \leq g-1$, so
the outer inequalities of~\eqref{equation1.2} are both satisfied.

Finally, we include an elementary lemma on the growth of certain
functions of $x$, which we will apply for $x$ a simple expression in
$q$ or $g$ as needed in our theorems below.

\begin{lemma}
\label{lemma2.0}
\begin{enumerate}
\item
The functions $x/(x-1)^2$ and $(1+x^{-1})^2/(1-x^{-1})$ are both
decreasing for $x>1$.
\item
The function $(1-2^{-x})^x$ is increasing for $x > 1/\log 2
\approx 1.443$.
\item
The function $\bigl[(1+2^{-x})/(1-2^{-x})\bigr]^x$ is decreasing
for $x > 1/\log 2$.
\end{enumerate}
\end{lemma}
\begin{proof}
Part~(1) is trivial.  Part~(2) follows from noting that the
logarithm of the function in question is 
$x\log(1-2^{-x}) = -\sum_{n \geq 1} x2^{-nx}/n$, and that the
functions $x2^{-nx}$ are all decreasing for $x > 1/\log 2$ (consider
the derivative of $x\exp(-ax)$).
Part~(3) is similar, using 
$x[\log(1+2^{-x}) - \log(1-2^{-x})] = 2 \sum_{n \text{ odd}} x2^{-nx}/n$.
\end{proof}

We can now get down to the business of our main result.  As our proof
relies on the Weil bounds for the zeta function and abelian
L-functions of $C$, we begin by fixing some more notation and
recalling the statements that we need.  For a survey of these results,
see Chapter~9 and the appendix of~\cite{Rosen}.

\begin{definition}
\label{definition2.1}
We define $\Div^d C$ to be the set of $\Fq$-rational divisors of
degree $d$ on $C$, and we define $\Eff^d C$ to be the subset of
effective divisors.  We write $J = (\Pic^0 C)(\Fq) = (\Div^0 C)/\sim$,
where $D \sim E$ means that $D$ and $E$ are linearly equivalent; we
write the equivalence class of $D$ as $[D]$.
\end{definition}

Since there is no period-index obstruction over finite fields, the
natural map $\Div^d C \to (\Pic^d C)(\Fq)$ is surjective.  Moreover, 
$\Eff^d C \to (\Pic^d C)(\Fq)$ is surjective if $d \geq g$.  We also
choose once and for all an $\Fq$-rational divisor $D_0$ of degree~$1$
(this is possible over a finite field) and use it to identify
$(\Pic^d C)(\Fq)$ with $J$ via $[D] \mapsto [D - d D_0]$.  We consider
the group $\Jhat$ of characters $\chi: J \to \C^*$.  Due to our
identification, we can evaluate $\chi$ on elements of 
$(\Pic^d C)(\Fq)$, or, for that matter, on $\Div^d C$.

\begin{definition}
\label{definition2.2}
We define $N_d = \abs{\Eff^d C}$, and for $\chi \in \Jhat$ we define
$N_{d,\chi} = \sum_{D \in \Eff^d C} \chi(D)$.  Thus $N_d = N_{d,1}$
for the trivial character $\chi = 1$.
\end{definition}

Note in the above definition that $N_{d,\chi}$ depends on the specific
choice of $D_0$, but that the effect of changing $D_0$ to $D'_0$ is to
multiply $N_{d,\chi}$ by the root of unity $\chi([D_0 - D'_0])^d$.  A
similar statement holds for the quantities $\alpha_{i,\chi}$ below.
Our main concern is bounds for the absolute value $\abs{N_{d,\chi}}$,
which is not affected by the choice of $D_0$.

\begin{theorem}[A. Weil]
\label{theorem2.3}
For $\chi \in \Jhat$, define the L-function $L_C(T,\chi)$ and, for the
trivial character, the zeta-function $Z_C(T)$ by
\begin{equation}
\label{equation2.1}
L_C(T,\chi) = \sum_{d \geq 0} N_{d,\chi} T^d,
\qquad
Z_C(T) = L_C(T,1) = \sum_{d \geq 0} N_d T^d.
\end{equation}
Then $L_C(T,\chi)$ is a polynomial for $\chi \neq 1$, while $Z_C(T)$
is a rational function.  These have the form
\begin{equation}
\label{equation2.2}
Z_C(T) = \frac{\prod_{i=1}^{2g}(1-\alpha_i T)}{(1-T)(1-qT)},
\qquad
L_C(T,\chi) = \prod_{i=1}^{2g-2} (1 - \alpha_{i,\chi}T),
\text{ for } \chi \neq 1,
\end{equation}
with all the $\abs{\alpha_i}$ and $\abs{\alpha_{i,\chi}}$ equal to
$q^{1/2}$.  Moreover, we have
\begin{equation}
\label{equation2.3}
\abs{J} = \prod_{i=1}^{2g} (1 - \alpha_i), 
\quad\text{ hence } (q^{1/2} - 1)^{2g} \leq \abs{J} \leq (q^{1/2} + 1)^{2g},
\end{equation}
\begin{equation}
\label{equation2.4}
\abs{C(\F_{q^d})} = q^d + 1 - \sum_{i=1}^{2g} \alpha_i^d.  
\end{equation}
\end{theorem}

We deduce the following bounds:
\begin{corollary}
\label{corollary2.4}
\begin{enumerate}
\item
For $\chi \neq 1$, we have $N_{d,\chi} = 0$ unless 
$0 \leq d \leq 2g-2$, in which case we have
$\abs{N_{d,\chi}} \leq \binom{2g-2}{d} q^{d/2} \leq 4^{g-1} q^{d/2}$.
\item
For $\chi = 1$, we have
$0 \leq N_d \leq (\frac{q}{q-1})(1 + q^{-1/2})^{2g} q^d$.
\item
The number $N_d^{irr}$ of irreducible (i.e., prime) effective divisors
of degree $d$ satisfies $N_d^{irr} \leq (q^d + 2gq^{d/2} + 1)/d$.
\end{enumerate}
\end{corollary}
\begin{proof}
We first prove statement (1).
Since $N_{d,\chi}$ is the coefficient of $T^d$ in $L_C(T,\chi)$, one
immediately obtains that $N_{d,\chi}$ is a sum of $\binom{2g-2}{d}$
terms, each of absolute value $q^{d/2}$.  Moreover, $\binom{2g-2}{d}
\leq (1+1)^{2g-2} = 4^{g-1}$.  (One can reduce this estimate by a factor
of roughly $g^{1/2}$, but our main concern is the exponential
dependence on $g$.) 
We now show statement (2).  The coefficient of $T^j$ in
$1/(1-T)(1-qT)$ is the positive number $(q^{j+1} - 1)/(q-1)$, which is
bounded above by $(\frac{q}{q-1})q^j$; rewriting $j = d-k$, the bound
becomes $(\frac{q}{q-1})q^{d-k}$.
In the other factor of $Z_C(T)$, which is $\prod_{i=1}^{2g}(1-\alpha_i
T)$, the coefficient of $T^k$ is bounded by $\binom{2g}{k} q^{k/2}$.
Taking the product with $1/(1-T)(1-qT)$, we obtain that $N_d$, being the
coefficient of $T^d$ in $Z_C(T)$, is bounded by
$N_d \leq (\frac{q}{q-1})\sum_{k=0}^d \binom{2g}{k} q^{d-(k/2)}$,
which easily gives the  bound in (2) (this works even if $d > 2g$).
Finally, statement (3) holds because each irreducible effective
divisor of degree $d$ can be regarded as a sum of $d$ distinct points
of $C(\F_{q^d})$, so we can apply~\eqref{equation2.4}.
\end{proof}

\begin{definition}
\label{definition2.5}
Fix $d \geq 0$.  We define the set 
$X_d \subset \Eff^d C \times \Eff^d C$
by
\begin{equation}
\label{equation2.5}
X_d = \{(D,E) \in \Eff^d C \times \Eff^d C \mid D \sim E \}.
\end{equation}
One can view $X_d$ as the set of $\Fq$-rational points of a subvariety of
$\Sym^d C \times \Sym^d C$.
\end{definition}

The set $X_d$ is closely related to the set of rational functions of
degree~$d$ on~$C$, as studied in~\cite{Elkies}, and our method of
bounding $\abs{X_d(\Fq)}$ in~Theorem~\ref{theorem2.6} below, via a sum
over the characters $\chi \in \Jhat$, was inspired by that article.

\begin{theorem}
\label{theorem2.6}
Assume that $0 \leq d \leq g-1$ and $q \geq 5$.  The number of points
of $X_d$ satisfies the inequality
\begin{equation}
\label{equation2.6}
\abs{X_d} \leq 16^g q^d.  
\end{equation}
\end{theorem}
\begin{remark}
\label{remark2.7}
The argument below is still valid for $q \leq 4$, but it yields an
order of growth of $\abs{X_d}$ that is bounded by $C^g q^d$, for $C$ a
constant larger than~$16$.  This is still enough to see that for all
$q$, the variety whose set of $\Fq$-rational points is $X_d$ has
dimension at most $d$.
\end{remark}
\begin{proof}
We can write $\abs{X_d}$ in terms of a sum over the
characters $\chi$ of $J$:
\begin{equation}
\label{equation2.7}
\begin{split}
\abs{X_d} &= \sum_{D, E \in \Eff^d C} 
\frac{1}{\abs{J}} \sum_{\chi \in \Jhat} 
   \chi(D) \, \chi^{-1}(E)\\
&= \frac{1}{\abs{J}} \sum_{\chi \in \Jhat} \quad 
  \sum_{D, E \in \Eff^d C} \chi(D) \, \chi^{-1}(E)
 = \frac{N_d^2}{\abs{J}}
     +  \frac{1}{\abs{J}}
        \sum_{1 \neq \chi \in \Jhat} N_{d,\chi} N_{d,\chi^{-1}}.\\
\end{split}
\end{equation}
The second term is bounded by $16^{g-1} q^d$, by
part~(1) of Corollary~\ref{corollary2.4},
and by the fact that $\abs{J} = \abs{\Jhat}$. 
The first term is bounded by 
$(\frac{q}{q-1})^2\bigl[(1 + q^{-1/2})^2/(1 - q^{-1/2})\bigr]^{2g} q^{2d-g}$,
by~\eqref{equation2.3} and part~(2) of Corollary~\ref{corollary2.4}.
Now, as $q$ increases, the quantities $q/(q-1)^2$ and 
$(1 + q^{-1/2})^2/(1 - q^{-1/2})$ 
both decrease, by part~(1) of Lemma~\ref{lemma2.0},
so they are bounded by $5/16$ and
$(1+1/\sqrt{5})^2/(1-1/\sqrt{5}) < 4$, respectively.  Thus the first term is
at most $[(5/16)4^{2g} q^{d-g+1}] \cdot q^d$, which in turn is bounded
by $(5/16) 16^g q^d$, because $d \leq g-1$.  Combining the two terms
completes the proof.
\end{proof}

We are now ready for the main result of this article.

\begin{theorem}
\label{theorem2.8}
Assume that $q \geq 5$, and that either (i) $0 \leq d \leq g-1$ and
$i \geq 1$, or (ii) $g-1 \leq d \leq 2g-2$ and $i-d-1+g \geq 1$.
Then the cardinality of the Brill-Noether locus~\eqref{equation1.1}
is bounded by
\begin{equation}
\label{equation2.8}
\abs{\{\LL \in \Pic^d C(\Fq) \mid h^0(\LL) \geq i\}} \leq 16^g q^{d-2i+2}.
\end{equation}
\end{theorem}
\begin{proof}
As we discussed just before~\eqref{equation2.0}, there is an
involution on the line bundles on $C$ exchanging $\LL$ with
$\KK \tensor \LL^{-1}$, and this involution exchanges conditions (i)
and~(ii) without changing the value of $d - 2i + 2$.  Thus we may
assume in the proof that condition~(i) holds, which allows us to use
Theorem~\ref{theorem2.6}.  Under condition~(i), consider the map
$f:X_d \to \Pic^d C(\Fq)$, given by $f(D,E) = [D] = [E]$.  The fibre
of $f$ over a point $\LL$ with $h^0(\LL) = \ell > 0$ is isomorphic to
$\Projective^{\ell-1}(\Fq)\times \Projective^{\ell-1}(\Fq)$. If $\ell
\geq i$, this fibre has at least $q^{2i-2}$ points, so, combining with
our estimate~\eqref{equation2.6}, we obtain the result.
\end{proof}

\begin{remark}
\label{remark2.9}
If $i=1$ above, it is better to bound the cardinality of the
Brill-Noether locus by $N_d$, since every $\LL$ with $h^0(\LL) \geq 1$
is represented by at least one effective divisor.
For $q \geq 5$, this gives a bound of $(5/4)(1+1/\sqrt{5})^{2g} \cdot
q^d < (5/4)(2.1)^g \cdot q^d$.  This is much better than $16^g q^d$
when $g \geq 1$.
\end{remark}

\begin{remark}
\label{remark2.9.5}
Theorem~\ref{theorem2.8} implies both Clifford's and Martens' theorems
over a finite field.  Namely, replacing $\F_q$ with $\F_{q^a}$ and
letting $a \to \infty$, we obtain that the dimension of the
Brill-Noether locus~\eqref{equation1.1} is bounded above by $d-2i+2$,
recovering Martens' theorem, and hence is empty whenever 
$d-2i+2 < 0$, recovering Clifford's theorem and the middle inequality
in~\eqref{equation1.2}.  In fact, \eqref{equation1.2} is equivalent to
combining $d-2i+2 \geq 0$ with [condition (i) or condition (ii)] above.
\end{remark}

For the next result, note that all line bundles of degree $g+1$ are
base point free if $g \leq 1$, so we have assumed that $g \geq 2$.
Also note that the requirement $q \geq 16^g$ can be weakened, but
we are anyhow only interested if our bound on the final
probability is less than~$1$.

\begin{theorem}
\label{theorem2.10}
Assume $g \geq 2$ and $q \geq 16^g$.  Then the probability that a
uniformly randomly chosen element of $\Pic^{g+1} C(\Fq)$ is not base point
free is at most $(16^g \cdot g)/q$.
\end{theorem}
\begin{proof}
The number of elements of $\Pic^{g+1} C(\Fq)$ is $\abs{J}$, which we can
bound by~\eqref{equation2.3}.  The point is to count the number of
line bundles $\LL \in \Pic^{g+1} C(\Fq)$ that are not base point free.  All
line bundles $\LL  \in \Pic^{g+1} C(\Fq)$ have $h^0(\LL) \geq 2$, but if
such a line bundle is not base point free, then there exists an
irreducible divisor $E$ (part of the divisor of base points of
$\LL$) for which $h^0(\LL(-E)) = h^0(\LL) \geq 2$.  Such an $E$ must
necessarily have $\deg E \leq g-1$, since otherwise we would have $\deg
\LL(-E) \leq 1$, which would force $h^0(\LL(-E)) \leq 1$.
Writing $\LL' = \LL(-E)$
(equivalently, $\LL = \LL'(E)$)
and $e = \deg E$, we can thus bound the number of
$\LL \in \Pic^{g+1} C(\Fq)$ that are not base point free by the number of
pairs $(E,\LL')$, where $E$ is irreducible, $\deg \LL' = g+1-e$, and
$h^0(\LL') \geq 2$, as $e$ ranges over all of $\{1,2,\dots,g-1\}$.  For a
given $e$, there are $N^{irr}_e$ choices of $E$. 
Also, there are by~\eqref{equation2.8} at most $16^g q^{g-1-e}$ choices
of $\LL'$.
Hence the total
number of $\LL$ that are not base point free is at most
\begin{equation}
\label{equation2.9}
\sum_{e=2}^{g-1} N_e^{irr}
\cdot 16^g q^{g-1-e}.
\end{equation}
By 
part~(3) of
Corollary~\ref{corollary2.4}, we can bound~\eqref{equation2.9} above
by 
\begin{equation}
\label{equation2.10}
\begin{split}
 & \sum_{e=1}^{g-1} \frac{1}{e}(q^e + 2g q^{e/2} + 1) 16^g q^{g-1-e}\\
 < \> & q^{g-1} \cdot 16^g 
   \left[ \sum_{e=1}^{g-1} \frac{1}{e}
           + 2g \sum_{e=1}^\infty \frac{q^{-e/2}}{e}
           + \sum_{e=1}^\infty \frac{q^{-e}}{e}
   \right].
\end{split}
\end{equation}
Estimating the harmonic sum, and using that $q \geq 16^g \geq 256$, we
see that the factor in square brackets is bounded by 
$1 + \log(g-1) - 2g\log(1-1/16) - \log(1-1/256) \approx 1 +
\log(g-1) + 0.13g + 0.004$, which is in turn at most $0.7g$ (recall
that this is for integers $g \geq 2$).
Thus our final probability is at most
\begin{equation}
\label{equation2.11}
\frac{q^{g-1}\cdot 16^g \cdot 0.7g}{\abs{J}} 
\leq \frac{16^g \cdot 0.7 g}{q (1-q^{-1/2})^{2g}}
\leq \frac{16^g \cdot 0.7 g}{q (1-4^{-g})^{2g}}
\leq \frac{16^g \cdot g}{q}.
\end{equation}
The above uses the fact that for $g \geq 2$, we have
$(1-4^{-g})^{2g} \geq (1 - 4^{-2})^4 \approx 0.77$, by part~(2)
of Lemma~\ref{lemma2.0}.
\end{proof}

\begin{remark}
\label{remark2.11}
If $g=2$ above, a much better bound is possible.  Indeed, in that case
the only possible value of $e$ is $e=1$, and in that case the only
possible $\LL'$ is $\LL' = \KK$, since one can see that $\KK \tensor
(\LL')^{-1}$ must have degree~0 and $h^0 \geq 1$.  Thus the bound for
the number of $\LL$ is
just $N_1^{irr}$, the number of points on $C$, which is bounded by
$q + 4 q^{1/2} + 1$; we can divide this by $\abs{J}$ to get our final
bound.  A further improvement can be obtained by keeping the
eigenvalues $\alpha_1, \dots, \alpha_4$ for Frobenius in both
$N_1^{irr}$ and $\abs{J}$, and optimizing the ratio
$N_1^{irr}/\abs{J}$ over the possible $\alpha_i$ satisfying the Weil
bounds and the Poincar\'e duality relations $\alpha_1 \alpha_2 =
\alpha_3 \alpha_4 = q$.
\end{remark}

\begin{remark}
\label{remark2.12}
If we are a bit more careful with our bound for $\binom{2g-2}{d}$ that
goes into part~(1) of Corollary~\ref{corollary2.4}, we can probably remove
the factor of $g$, but this does not seem worth the effort here.
\end{remark}

We conclude this section with a precise bound on how likely it is that
a line bundle of degree $g-1$ has any global sections, under the same
hypotheses as Theorem~\ref{theorem2.10}.

\begin{proposition}
\label{proposition2.13}
Assume $g \geq 2$ and $q \geq 16^g$.  Then the probability that a
uniformly randomly chosen element $\LL \in \Pic^{g-1} C(\Fq)$ has $h^0(\LL)
\geq 1$ is at most $1.7/q$.
\end{proposition}
\begin{proof}
Using Remark~\ref{remark2.9} and our various estimates above, we see
that this probability is bounded above by
\begin{equation}
\label{equation2.12}
\frac{N_{g-1}}{\abs{J}}
\leq
\frac{(\frac{q}{q-1})(1 + q^{-1/2})^{2g} q^{g-1}}{(q^{1/2} - 1)^{2g}}
=
\frac{1}{q} \cdot
\frac{q}{q-1} \cdot \left[\frac{1+q^{-1/2}}{1-q^{-1/2}}\right]^{2g}.
\end{equation}
The fraction $q/(q-1)$ is at most $256/255 \approx 1.004$, while the
quantity in square brackets is bounded above by
$[(1+4^{-g})/(1-4^{-g})]^{2g}$, which, for $g \geq 2$, does not exceed
its value at $g=2$, by part~(3) of Lemma~\ref{lemma2.0}.  Now the
value at $g=2$ is approximately $1.65$, so the constant in our upper
bound is approximately $(1.004)(1.65)$, which is less than $1.7$.
\end{proof}

\section{Application to typical divisors}
\label{section3}

We apply our results from Section~\ref{section2} to bound the number
of typical divisors on a curve, in the sense of~\cite{KKMtypical}.  We
begin by recalling the definition of a typical divisor from that
article.  

\begin{definition}
\label{definition3.1}
We first set up the context in which we work.  From now on, $C$ comes
equipped with a distinguished rational point $\Pinf \in C(\Fq)$, and
the definition of a typical divisor on $C$ depends on the choice of
$\Pinf$.  We use the point $\Pinf$ (viewed as a divisor $D_0 = \Pinf$
of degree~$1$) to identify $\Pic^d C$ with $\Pic^0 C$ whenever
convenient, as we did just after Definition~\ref{definition2.1}.
\begin{enumerate}
\item
A divisor $D$ on $C$ is called good if $D$ is effective,
$\Fq$-rational, and disjoint from $\Pinf$.  We will assume that
the degree $d = \deg D$ satisfies $d \geq g$, and will also refer to
the corresponding line bundle $\LL = \OC(d\Pinf - D) \in \Pic^0 C(\Fq)$.
\item
We define the following $\Fq$-vector spaces that will appear
frequently in this section: for $N \in \Z$,
\begin{equation}
\label{equation3.1}
\begin{split}
W^N &= H^0(C,\OC(N\Pinf)),\\ 
W^N_D &= H^0(C,\OC(N\Pinf - D)) = H^0(C,\LL((N-d)\Pinf)).\\
\end{split}
\end{equation}
We will usually view $W^N$ and $W^N_D$ as Riemann-Roch spaces, hence as
subsets of the function field $\Fq(C)$; for example,
$W^N_D \subset W^{N+1}_D$. 
\item
A good divisor $D$ is called semi-typical if $W^{d+g-1}_D = 0$, or
equivalently if $H^0(C,\LL((g-1)\Pinf)) = 0$.
\item
A good divisor $D$ is called typical if there
exist $s \in W^{d+g}_D = H^0(C,\LL(g\Pinf))$ and
$t \in W^{d+g+1}_D = H^0(C,\LL((g+1)\Pinf))$, satisfying
\begin{equation}
\label{equation3.2}
s W^{2g} + t W^{2g-1} + W^{d+g-1} = W^{d+3g}.
\end{equation}
\end{enumerate}
\end{definition}

The above definition of semi-typicality depends only on the
corresponding line bundle $\LL$ (via its twist $\LL((g-1)\Pinf)$, and
not on the particular good divisor $D$, nor on its degree $d$.
In~\cite{KKMtypical}, we also proved that typicality of a divisor also
depends only on $\LL$.  The following is an equivalent
characterization of typicality in terms of $\LL$.

\begin{proposition}
\label{proposition3.2}
Let $D$ be a good divisor of degree $d \geq g$, with $g \geq 1$.  Then
$D$ is typical if and only if the following conditions hold for the
associated line bundle $\LL = \OC(d\Pinf - D)$:
\begin{enumerate}
\item
The line bundles $\LL$ and $\LL^{-1}$ are semi-typical, i.e.,
$h^0(\LL((g-1)\Pinf)) = h^0(\LL^{-1}((g-1)\Pinf)) = 0$.
\item
The line bundle $\LL((g+1)\Pinf)$ is base point free.
\end{enumerate}
\end{proposition}
\begin{proof}
The fact that the above conditions are equivalent to typicality is
implicit in~\cite{KKMtypical}, but in the interest of making this
article self-contained we will include a streamlined version of the
argument.

We first show that the conditions above imply that $\LL$ is typical.
We know from semi-typicality that $W^{d+g-1}_D = 0$.
To find our desired $s,t \in W^{d+g+1}_D$, we observe that 
$\dim W^{d+g+1}_D \geq 2$ by Riemann-Roch, but that in the inclusions
$0 = W^{d+g-1}_D \subset W^{d+g}_D \subset W^{d+g+1}_D$, the
dimensions increase by at most~$1$.  Thus $\dim W^{d+g}_D = 1$ and
$\dim W^{d+g+1}_D = 2$, and we choose accordingly any generator
$s \in W^{d+g}_D - \{0\}$, and any $t \in W^{d+g+1}_D - W^{d+g}_D$,
from which it follows that $\{s,t\}$ is a basis for
$W^{d+g+1}_D = H^0(C,\LL((g+1)\Pinf))$.
Viewing $s$ and $t$ as elements of the function field $\Fq(C)$, this
means moreover that $s$ and $t$ have poles at $\Pinf$ of exact orders
$d+g$ and $d+g+1$, respectively.  Hence their divisors have the form
\begin{equation}
\label{equation3.3}
\Divisor s = -(d+g)\Pinf + D + A,\qquad
\Divisor t = -(d+g+1)\Pinf + D + B,
\end{equation}
where $A$ and $B$ are good divisors of degrees $g$ and $g+1$,
respectively.  Moreover, the fact that $\LL((g+1)\Pinf) =
\OC((d+g+1)\Pinf - D)$ is base point free means that $A$ and $B$ are
disjoint.
Since $A$ and $B$ are disjoint, it follows that 
$sW^{2g} \intersect tW^{2g-1}
 = W^{d+3g}_{D+A} \intersect W^{d+3g}_{D+B}
 = W^{d+3g}_{D+A+B}$.
This last space is isomorphic, via division by $st$, to
$H^0(C,\OC((-d+g-1)\Pinf+D) = H^0(C,\LL^{-1}((g-1)\Pinf) = 0$, by
semi-typicality of $\LL^{-1}$.  Hence
$\dim(sW^{2g} + tW^{2g-1}) = (g+1) + g = \dim W^{d+3g}_D$.  On the
other hand,
$sW^{2g} + tW^{2g-1}  = W^{d+3g}_{D+A} + W^{d+3g}_{D+B} \subset W^{d+3g}_D$,
and we deduce that $sW^{2g} + tW^{2g-1} = W^{d+3g}_D$.  This implies that
$(sW^{2g} + tW^{2g-1}) \intersect W^{d+g-1} = W^{d+g-1}_D = 0$ (since
$\LL$ is semi-typical), hence that
$\dim (sW^{2g} + tW^{2g-1} + W^{d+g-1})
= \dim W^{d+3g}_D + \dim W^{d+g-1} = (2g+1) + d = \dim W^{d+3g}$, and
we have proved~\eqref{equation3.2}.  Hence $\LL$ is typical, as desired.

We now prove the converse, namely that if $D$ is typical, then $\LL$
satisfies the above conditions.  This is fairly close to running the
above argument backwards, but needs some details to be filled in.
From~\eqref{equation3.2}, we deduce by counting dimensions that the
sum is direct (and, even before that, that $s$ and $t$ are nonzero).
Hence
\begin{equation}
\label{equation3.4}
\begin{split}
sW^{2g} \intersect tW^{2g-1} &= 0,\\
\dim (sW^{2g} + tW^{2g-1}) &= 2g+1 = \dim W^{d+3g}_D, \\
(sW^{2g} + tW^{2g-1}) \intersect W^{d+g-1} & = 0.\\
\end{split}
\end{equation}
Since we anyhow have $sW^{2g} + tW^{2g-1} \subset W^{d+3g}_D$, we
deduce that these spaces are equal; from the last intersection above,
it follows that $W^{d+g-1}_D = 0$, so $D$ is semi-typical.  Moreover,
from $sW^{2g} + tW^{2g-1} = W^{d+3g}_D$, it follows that $s$ and
$t$, when viewed as sections of
$\LL((g+1)\Pinf) = \OC((d+g+1)\Pinf - D)$, cannot have 
any common vanishing at an irreducible effective divisor $E$ (which
could be $\Pinf$), since we would then have 
$sW^{2g} + tW^{2g-1} \subset W^{d+3g}_{D+E} \subsetneq W^{d+3g}_D$.
This shows that $\LL((g+1)\Pinf)$ is base point free, and that the
divisors of $s$ and 
$t$ are as in~\eqref{equation3.3}; this uses the fact that $D$ is
semi-typical for the precise behavior at $\Pinf$.  Finally, the fact
that $sW^{2g} \intersect tW^{2g-1} = 0$ implies that $\LL^{-1}$ is
semi-typical, because the existence of a nonzero 
$u \in H^0(C,\OC((-d+g-1)\Pinf + D)$ would give rise to a nonzero
$ust = s(ut) = t(us) \in sW^{2g} \intersect tW^{2g+1}$.
\end{proof}

The above result allows us to view typicality and semi-typicality as
properties of a line bundle $\LL \in \Pic^0 C(\Fq)$.  In this context,
we immediately obtain a bound on how likely it is that a uniformly
randomly chosen line bundle $\LL$ is not typical or semi-typical.  If
$g=1$, then all nontrivial line bundles $\LL \not\isomorphic \OC$ are
easily seen to be typical.  We therefore limit ourselves to $g \geq 2$
and moderately large $q$, and obtain our desired bound.

\begin{theorem}
\label{theorem3.3}
Assume $g \geq 2$ and $q \geq 16^g$.  Then the probability that a
uniformly randomly chosen element $\LL \in \Pic^0 C(\Fq)$ is not
typical is at most $(16^g \cdot g + 3.4)/q$.  The probability that
$\LL$ is not semi-typical is at most $1.7/q$.
\end{theorem}
\begin{proof}
Since $\LL \in \Pic^0 C(\Fq)$ is chosen uniformly at random, we see
that also $\LL(N\Pinf) \in \Pic^N C(\Fq)$ is chosen uniformly at
random for any $N$, as is $\LL^{-1}(N\Pinf)$.  The statement on
semi-typicality is now  
Proposition~\ref{proposition2.13}, applied to $\LL((g-1)\Pinf)$.  The 
statement on typicality follows from bounding the probability that at
least one of the conditions in Proposition~\ref{proposition3.2} fails,
using Proposition~\ref{proposition2.13} applied to $\LL((g-1)\Pinf)$
and $\LL^{-1}((g-1)\Pinf)$, and Theorem~\ref{theorem2.10} applied to
$\LL((g+1)\Pinf)$.
\end{proof}

As a small final observation, we note that the probability that at
least one of $\LL$ and $\LL^{-1}$ in the above theorem is not typical
is bounded by $(16^g \cdot 2g + 3.4)/q$, due to the probability that
$\LL^{-1}((g+1)\Pinf)$ is not base point free.





\providecommand{\bysame}{\leavevmode\hbox to3em{\hrulefill}\thinspace}
\providecommand{\MR}{\relax\ifhmode\unskip\space\fi MR }
\providecommand{\MRhref}[2]{%
  \href{http://www.ams.org/mathscinet-getitem?mr=#1}{#2}
}
\providecommand{\href}[2]{#2}

\end{document}